\documentclass[conference]{IEEEconf}

\IEEEoverridecommandlockouts                              

\usepackage{graphicx} 
\usepackage{amsmath} 
\usepackage{amssymb}  
\usepackage{amsthm}
\usepackage{color}

\newtheorem{theorem}{Theorem}
\newtheorem{proposition}{Proposition}

\newtheorem{assumption}{Assumption}
\newtheorem{lemma}{Lemma}
\newtheorem{definition}{Definition}

\newcommand{\R}{{\mathbb R}}
\newcommand{\N}{{\mathbb N}}
\newcommand{\Ltwo}{{\mathbf L^2}}

\title{
Feedback design of spatially-distributed filters with tunable resolution
}

\author{Alessio Franci
\thanks{This work was supported by UNAM-DGAPA-PAPIIT grant IN102420 and by CONACyT grant A1-S-10610.}
\thanks{Alessio Franci is with the Department of Mathematics, National Autonomous University of Mexico, Ciudad Universitaria, 04510, Mexico City, Mexico
        {\tt\small afranci@ciencias.unam.mx}}%
}

\begin{document}

\maketitle

\begin{abstract}
	We derive gain-tuning rules for the positive and negative spatial-feedback loops of a spatially-distributed filter to change the resolution of its spatial band-pass characteristic accordingly to a wavelet zoom, while preserving temporal stability.
	The filter design is inspired by the canonical spatial feedback structure of the primary visual cortex and is motivated by understanding attentional control of visual resolution. Besides biology, our control-theoretical design strategy is relevant for the development of neuromorphic multiresolution distributed sensors through the feedback interconnection of elementary spatial transfer functions and gain tuning.
	
\end{abstract}

\thispagestyle{empty}
\pagestyle{empty}

\section{Introduction}

Our visual system is able to let us ``see the forest and the trees''~\cite{Diamond2018}, meaning that the fine local details and the coarse global structure of a scene are perceived simultaneously. In other words, our visual systems processes the visual world simultaneously at multiple scales. Psychophysical~\cite{anton2013attentional} and electrophysiological~\cite{bredfeldt2002dynamics} evidences suggest that the ability of perceiving the visual world in a multi-scale fashion might rely on a continuous adjustment of the spatial resolution at which incoming visual stimuli are filtered by the first layers of the visual system. Attention is key in driving modulation of visual resolution~\cite{shelchkova2020modulations}. The same location in the visual field can be processed at high or low resolution depending on whether the attentional focus is directed toward it or not. The transition between high-resolution and low-resolution processing is seemingly continuous, both in time and in space.

Recently~\cite{franci2021cosyne}, we suggested a formal analogy between attention-driven visual resolution changes and a {\it wavelet zoom}, that is, the progressive zooming-in into the structure of a signal through a continuous wavelet transform at increasingly finer scales~\cite{mallat1999wavelet}. We showed numerically in a linear neural field model that the canonical local-excitation/lateral-inhibition feedback structure of the primary visual cortex~\cite{Wilson1973} can focus (i.e., increase the resolution of) the feedforward visual kernel accordingly to a wavelet zoom. Crucially, our model does not require any unrealistic finely-tuned, space-localized scaling of synaptic connectivity, as a feedforward model would require. A space-homogeneous upscale of excitatory and inhibitory feedback connections is sufficient to realize the wavelet zoom, robustly to parameter uncertainties and spatial heterogeneities.

In this paper, we illustrate the theory and key ideas underlying our numerical results. We prove that in a spatially-distributed feedback filter, with fixed feedforward and feedback spatial transfer functions, a balanced (i.e., proportional) scaling of the positive and negative spatial-feedback gains realizes the whole space-frequency dictionary generated by the feedforward kernel while preserving the temporal stability of the filter. In other words, we show that it is possible to change the resolution of a spatially-distributed filter accordingly to a wavelet zoom by gain-tuning and, crucially, without changing its feedforward and feedback convolution kernels, robustly to parameter uncertainties and heterogeneities.

Our work is similar in spirit to classical works on the analog realization of wavelet transforms (see, e.g.,~\cite{haddad2005log}) but with some key differences. First, time is replaced by space, that is, capacitors are replaced by spatially-distributed, hard-wired connections. Second, we do not aim at realizing a spatial-filter bank, but rather at designing a spatially-distributed system whose transfer function can continuously be modulated in resolution accordingly to a wavelet transform and through a few tuning parameters. Third, our approach is neuromorphic, that is, we do not aim at fitting existing computational wavelets. Rather, our wavelets are those that arise from the feedforward and feedback structure of the first layers of the visual system. The result is a design methodology that relies on the feedback interconnection of elementary spatial transfer functions, easily implementable in practice without the need of any fine tuning. In particular, our methodology is compatible with the sloppiness of analog hardware because the lack of any fine tuning makes it naturally robust to parameter uncertainties and other practical approximations. It is a new candidate for the design of neuromorphic multiresolution analog visual sensors inspired by attention mechanisms in the primary visual cortex~\cite{indiveri2019importance}.

The paper is organized as follows. Section~\ref{SEC: notation} introduces the needed notations and definitions. In Section~\ref{SEC: V1 feedback filter}, we derive a spatial-feedback model of the primary visual cortex and we construct its closed-loop spatial transfer function. In Section~\ref{SEC: loop shaping real}, we use frequency-domain methods to derive the gain-tuning rules to robustly realize a wavelet zoom in the closed-loop transfer function and illustrate our theoretical results via a numerical example. A discussion and future directions are presented in Section~\ref{SEC: discussion}.

\section{Notation and definition}
\label{SEC: notation}

$\R$ denotes the set of real numbers and $\R_+$ denotes the set of positive real numbers. $\Ltwo(\R)$ denotes the space of square-integrable (finite energy) functions, i.e., $f\in\Ltwo(\R)$ if and only if
$
\int_\R|f(x)|^2dx<\infty\,.
$
$\Ltwo(\R)$ is equipped with the norm
$
\|f\|=\left(\int_\R|f(x)|^2dx\right)^{1/2}
$.
Given $f,g\in\Ltwo$, their {\it convolution} is defined
as
\[
f\star g(x)=\int_\R f(u) g(x-u)du
\]

Given $f\in\Ltwo$, its {\it Fourier transform} is defined as
\[\mathcal F(f)(\lambda)=\int_\R f(x)e^{-i\lambda x}dx=:\hat f(\lambda)\,.\]
The {\it inverse Fourier transform is defined as} 
\[\mathcal F^{-1}(\hat f)(x)=\frac{1}{2\pi}\int_\R \hat f(\lambda) e^{i\lambda x}d\lambda=:f(x)\]
Given $f:\R^2\to\R$ such that $f(\cdot,t)\in\Ltwo(\R)$ for all $t$, we also define
$$\hat f(\lambda,t)=\int_\R e^{-i\lambda x} f(x,t) dx\,.$$

The following definitions are borrowed from~\cite{mallat1999wavelet}.
\begin{definition}
	A function $f:\R\to\R$ is said to have a \emph{fast decay} if for any $m\in\N$ there exists $C_m>0$ such that
	$$\forall x\in\R,\quad |f(x)|\leq \frac{C_m}{1+|x|^m}. $$
\end{definition}
\begin{definition}[Mallat]\label{DEF: wavelet}
	A \emph{wavelet} is a function $\psi\in\Ltwo(\R)$ such that it has zero average, i.e., $\int_\R \psi(t)dt=0$ and is normalized, i.e., $\|\psi\|=1$.
\end{definition}
\noindent Here, we will only consider real wavelets $\psi:\R\to\R$.

\begin{definition}
	The \emph{dictionary of space-frequency atoms} generated by a wavelet $\psi$ is the set
	\[
	\mathcal D_{ff}=\left\{\psi_{u,s}(x)=\frac{1}{\sqrt{s}}\psi\left(\frac{x-u}{s}\right)\right\}_{u\in\R,s\in\R_+}\,.
	\]
\end{definition}

\begin{definition}
	Given $f\in\Ltwo$ and a real wavelet $\psi$, the \emph{wavelet transform} of $f$ is defined as
	\[
	Wf(u,s)=\int_\R f(x)\psi_{u,s}(x)dx=f\star\bar \psi_{s}(u)\,,
	\]
	where $\bar\psi_s(x)=\psi_{0,s}(-x)$.
\end{definition}
Because by definition $\mathcal F(\bar\psi_s)(0)=0$, a wavelet transform performs a multi-scale band-pass filtering.

\section{V1 as an excitatory/inhibitory feedback spatial filter}
\label{SEC: V1 feedback filter}

\subsection{A linear neural field description of V1}

The primary visual cortex (V1) codes incoming visual stimuli into the electrical activity of its millions of neurons. Neural fields~\cite{coombes2014neural} constitute a useful spatially-distributed mathematical description of the rich spatiotemporal dynamics of the primary visual cortex. Let $a(x,t)\in\R$ represent the electrical activity of an infinitesimal patch of the visual cortex located at position $x$ and at time $t$. Then
\begin{equation}\label{EQ: V1 neural field generic}
a_t(x,t)=-\gamma a(x,t)+\int_{\Omega} w(x,y) \sigma(a(y,t)) dy + \tilde h(x,t),
\end{equation}
where $a_t=\frac{\partial a}{\partial t}$, $\gamma\geq 0$ is the damping coefficient of V1 electrical activity (biologically related to V1 neurons' input conductance), $\Omega$ is the spatial domain of the visual cortex, $w(x,y)$ defines how and how strongly neurons at position $x$ are affected by neurons at position $y$, $\sigma:\R\to\R$ is a monotone increasing function modeling neuronal synaptic interactions and $\tilde h(x,t)$ represent exogenous inputs at location $x$ and time $t$. In this paper we are interested in the (linear) spatial-filtering properties of a the primary visual cortex. We thus assume $\sigma(a)=a$. We will come back to the role of nonlinearities in neuronal interactions in Section~\ref{SEC: discussion}.

A useful mathematical simplification is to assume that the spatial domain $\Omega=\R$. Under this simplification, another natural, useful and commonly made assumption is that the spatio-temporal dynamics are homogeneous (i.e., spatially-invariant), that is, $w(x,y)=W(x-y)$. Then, the spatial part of model~\eqref{EQ: V1 neural field generic} can be rewritten in spatial convolution form as
\begin{equation}\label{EQ: V1 neural field generic conv}
a_t=-\gamma a+W\star a+\tilde h.
\end{equation}
The input $\tilde h(x,t)$ to V1 is obtained by filtering the actual visual input to the animal $h(x,t)$ through the visual layers upstream V1. We let $W_{ff}$ the convolution kernel associated to these upstream neural layers, i.e., $\tilde h(x,t)=(W_{ff}\star \tilde h(\cdot,t))(x)$. For obvious reasons, we call $W$ the {\it feedback} kernel and $W_{ff}$ the {\it feed-forward} kernel associated to V1.

\begin{assumption}\label{AS: W in L2}
$W,W_{ff}\in\Ltwo(\R)$.
\end{assumption}
\noindent Under Assumption~\ref{AS: W in L2}, model~\eqref{EQ: V1 neural field generic conv} defines a linear space-invariant integro-differential operator on $\Ltwo(\R)$, which can be diagonalized in the Fourier basis by taking the spatial Fourier transform of both sides of the equation
\begin{equation}\label{EQ: V1 neural field generic fourier}
\hat a_t(\lambda,t)=-\gamma \hat a(\lambda,t)+\hat W(\lambda)\, \hat a(\lambda,t)+\hat W_{ff}(\lambda)\, \hat h(\lambda,t).
\end{equation}

\subsection{Excitatory and inhibitory feedback kernels}

Neuronal interactions inside V1 are both excitatory and inhibitory. It follows that $W(x-y)$ can both be positive or negative depending on whether a majority of neural projections between neurons at relative position $x-y$ are excitatory or inhibitory. Both types of projections decay exponentially in space, i.e., given a pair of neurons at distance $r$, the probabilities $p_{conn,E}(r)$ and $p_{conn,I}(r)$ of an excitatory or inhibitory connection between them are given by~\cite{boucsein2011beyond}
\begin{equation}\label{EQ: exp connection probability}
p_{conn,E}(r)=\frac{1}{2r_E}e^{-\frac{r}{r_E}},\quad p_{conn,I}(r)=\frac{1}{2r_I}e^{-\frac{r}{r_I}}\,,
\end{equation}
where $r_E>0$ and $r_I>0$ are the characteristic spatial scales of recurrent (feedback) excitatory and inhibitory interactions.

The feedback kernel $W$ can thus naturally be modeled as $W=K_EW_E-K_IW_I$,
where $K_E\geq0$ is the excitatory connection gain, $K_I\geq 0$ is the inhibitory connection gain,
\begin{equation}\label{EQ: exp connection bound exc}
|W_E(x-y)|\leq \frac{1}{2r_E}e^{-\frac{|x-y|}{r_E}}
\end{equation}
and
\begin{equation}\label{EQ: exp connection bound inh}
|W_I(x-y)|\leq \frac{1}{2r_I}e^{-\frac{|x-y|}{r_I}}
\end{equation}
Exponential decay of recurrent excitatory and inhibitory connections means that the kernels $W_E$ and $W_I$ are localized in space.

\subsection{The feedforward kernel}

Feedforward connections from upstream visual system to V1 (i.e., from the retina and through the visual thalamus) are also both excitatory and inhibitory. Feedforward projections also follow an exponential probability distribution~\eqref{EQ: exp connection probability} as a function of the distance between the projecting and receiving neuron. We thus assume
\begin{equation}\label{EQ: ff spatial inv}
 |W_{ff}(x-y)|\leq \frac{1}{2r_{ff}}e^{-\frac{|x-y|}{r_{ff}}}\,,
\end{equation}
where $r_{ff}>0$ is the characteristic spatial scale of feedforward connections. Furthermore, because visual neurons are only sensitive to spatial contrast variations, i.e., spatially uniform visual stimuli do not elicit a neural response, the feedforward visual kernel has zero average. For mathematical convenience and without loss of generality, we assume that $W_{ff}$ has been $\Ltwo$-normalized, which leads to
\begin{equation}\label{EQ: ff zero average + normalized}
\hat W_{ff}(0)=0,\quad \|W_{ff}\|=1´\,.
\end{equation}

\subsection{The spatial transfer function of V1}

Given a static visual stimulus $h(x,t)\equiv l(x)$, $l\in\Ltwo(\R)$, the equilibrium solution $a^*(x)$ of model~\eqref{EQ: V1 neural field generic conv} can easily be found in its frequency domain representation~\eqref{EQ: V1 neural field generic fourier} and reads
\begin{equation}\label{EQ: gen equilibrium solution}
\hat a^*(\lambda)=\frac{\hat W_{ff}(\lambda)}{\gamma-K_E\hat W_E(\lambda)+K_I\hat W_I(\lambda)}\hat l(\lambda)\,.
\end{equation}
This solution is stable if and only if all Fourier modes of~\eqref{EQ: V1 neural field generic fourier} are stable~\cite{Bamieh2002}, that is, if and only if
\begin{equation}\label{EQ: gen equilibrium solution stability}
-\gamma+\hat W(\lambda)<0,\quad \forall \lambda\in\R\,.
\end{equation}
Assuming the stability condition~\eqref{EQ: gen equilibrium solution stability} is satisfied, equation~\eqref{EQ: gen equilibrium solution} defines a spatial transfer function
\begin{equation}\label{EQ: neural field TF}
\hat H_{\gamma,K_E,K_I}(\lambda)=\frac{\hat W_{ff}(\lambda)}{\gamma-K_E\hat W_E(\lambda)+K_I\hat W_I(\lambda)}
\end{equation}
that characterizes the steady-state spatial filtering properties of our model V1. Figure~\ref{FIG: block diagram} show a block-diagram realization of $\hat H$. It highlights the (spatial) feedback nature of horizontal connections. Horizontal excitatory connections provides (spatial) posititve feedback. Horizontal inhibitory connections provide (spatial) negative feedback. The role of these feedback loops is to shape the kernel through which the neuronal layer filters the incoming signal $l$. For $K_E=K_I=0$, the open-loop layer's kernel is the feed-forward kernel $W_{ff}$. For non-zero feedback gains, the closed-loop kernel is reshaped by spatial feedback. Observe that the only tunable parameters in~\eqref{EQ: neural field TF} are the feedback gains $K_E$ and $K_I$, and the damping $\gamma$. The feed-forward $W_{ff}$ and feedback $W_E,W_I$ kernels are fixed.

\begin{figure}
\centering
\includegraphics[width=0.3\textwidth]{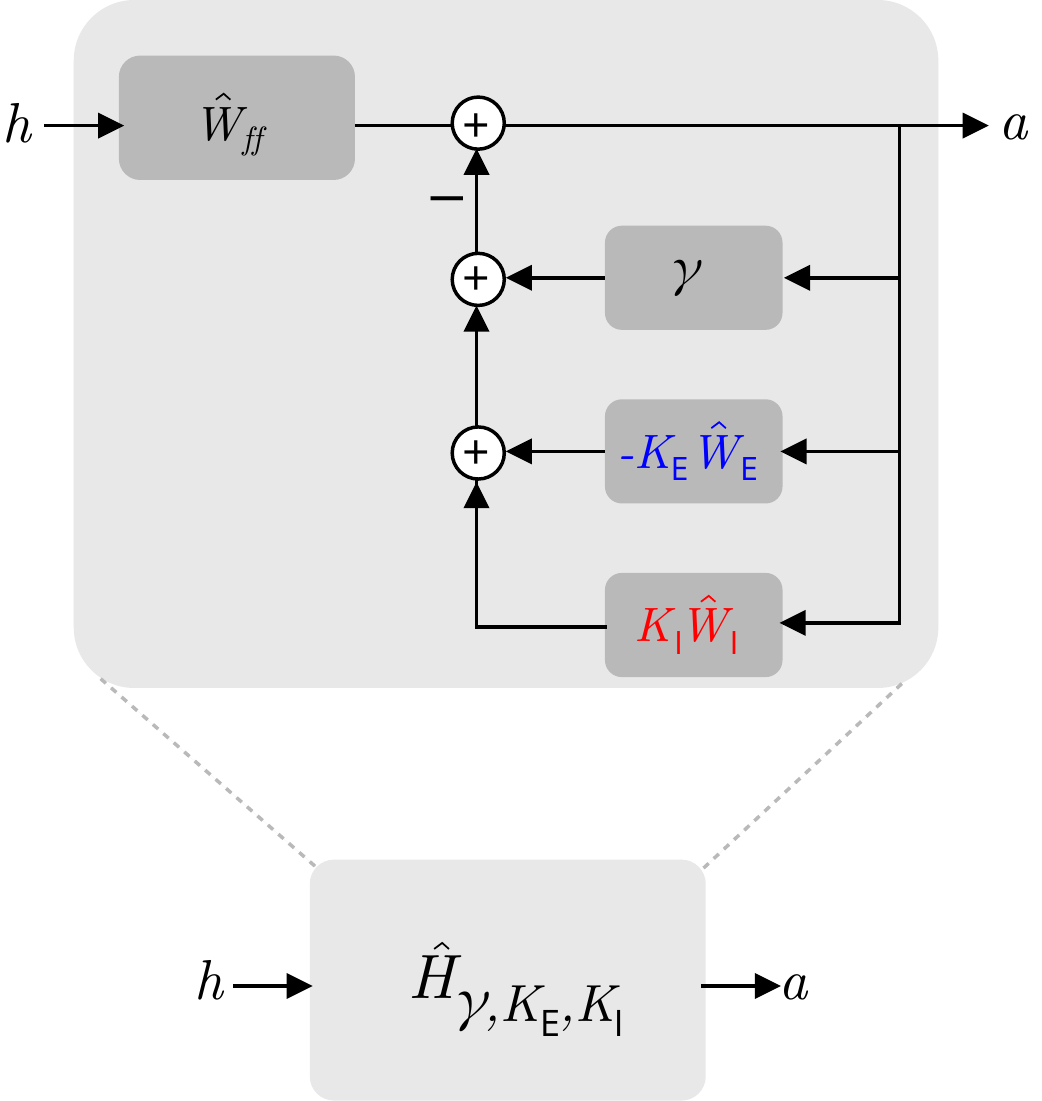}
\caption{Block diagram of the closed-loop spatial transfer function $\hat H_{\gamma,K_E,K_I}(\lambda)$.}\label{FIG: block diagram}
\end{figure}

\section{Feedback realization of a wavelet zoom}
\label{SEC: loop shaping real}

In this section we use frequency-domain close-loop analysis to illustrate how to realize a whole space-frequency dictionary through the modulation of the damping coefficient, and the excitatory and inhibitory feedback gains in the closed-loop transfer function~\eqref{EQ: neural field TF}. In our construction, the feedfoward kernel $W_{ff}$ defines the mother wavelet, which generate the space-frequency atom dictionary.
We start by illustrating these results in the case in which feedback and feedforward interactions are isotropic, that is, they solely depend on the distance between two neurons, i.e, all kernels are even functions of $x-y$. Generalizations are discussed in Section~\ref{SEC: discussion}.

\subsection{An exponentially-decaying isotropic feedforward kernel defines the mother wavelet}

A natural choice to represent the feedforward visual kernel while respecting the exponential decay of both excitatory and inhibitory feedforward connections~\eqref{EQ: exp connection probability} is
\begin{equation}\label{EQ: ff kernel shape}
W_{ff}(x-y)=\frac{1}{N(a,b)}\left(\frac{1}{2a}e^{-\frac{|x-y|}{a}}-\frac{1}{2b}e^{-\frac{|x-y|}{b}}\right)
\end{equation}
where $0<a<b$ and the $\Ltwo$-normalization factor $N (a, b) = \sqrt{\frac{(b - a)^2}{64ab(a + b)}} $. As illustrated in Figure~\ref{FIG: mother wavelet}, the feedforward kernel is made of a narrower excitatory part, with spatial scale $a$, and a broader inhibitory part, with spatial scale $b$.
This choice can be seen as an exponential approximation of center-surround receptive fields~\cite{alonso2009receptive}. The advantage of using an exponential approximation is that it leads to the rational transfer function
\begin{equation}\label{EQ: ff kernel fourier}
\hat W_{ff}(\lambda)=\frac{1}{N(a,b)}\frac{(b^2-a^2)\lambda^2}{(1+a^2\lambda^2)(1+b^2\lambda^2)}
\end{equation}
which is amenable to closed-loop algebraic manipulation, much in the same way as temporal transfer function design.

\begin{figure}
	\centering
	\includegraphics[width=0.3\textwidth]{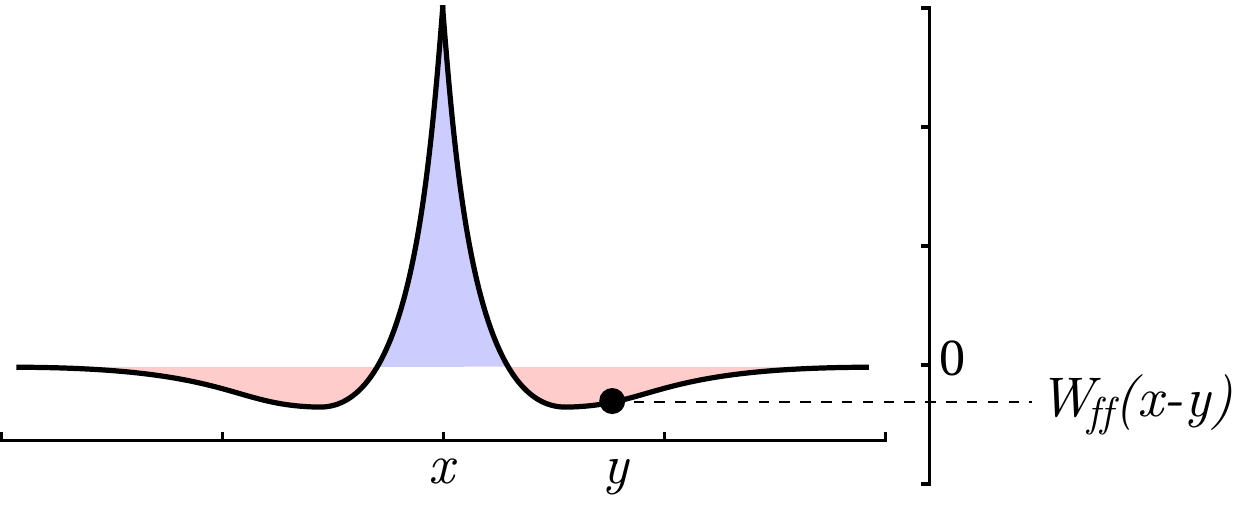}
	\caption{Qualitative shape of the feedforward kernel $W_{ff}(x-y)$. Red-shaded regions indicate inhibitory interactions. Blue-shaded regions indicate excitatory interactions.}\label{FIG: mother wavelet}
\end{figure}

The following proposition shows that $W_{ff}$ is a wavelet and that, moreover, it can be written as the second derivative of a fast-decaying function with non-zero average.
\begin{proposition}\label{PROP: Wff is wavelet}
	$W_{ff}$ as defined in~\eqref{EQ: ff kernel shape} is a wavelet with fast decay and, moreover, there exist a function with fast decay $\theta$, such that $\int_\R\theta(x)dx\neq 0$ and $W_{ff}=\theta''$.
\end{proposition}
\begin{proof}
	Invoking~\eqref{EQ: ff zero average + normalized}, $W_{ff}$ is a wavelet by construction and it has fast decay because it decays exponentially as $x\to\pm\infty$. To show the second part of the statement, recall that for any $f\in\Ltwo(\R)\cap C^2(\R)$, such that $f''\in\Ltwo(\R)$, $\mathcal F(\theta'')=-\lambda^2\mathcal F(\theta)$, and observe that $W_{ff}=\theta''$ with $\theta=\mathcal F^{-1}\left(\frac{-1}{N(a,b)}\frac{(b^2-a^2)}{(1+a^2\lambda^2)(1+b^2\lambda^2)}\right)\,.$
	The fact that $\theta$ has the fast decay property follows by the fact that $\mathcal F(\theta)$ is smooth and~\cite[Theorem~2.5]{mallat1999wavelet}.
\end{proof}

\subsection{A wavelet zoom defined by the feedforward kernel}

The dictionary of space-frequency atoms generate by the wavelet is $W_{ff}$ is 
\begin{equation}\label{EQ: ff dictionary}
\mathcal D_{ff}=\left\{\psi^{ff}_{u,s}(x)=\frac{1}{\sqrt{s}}W_{ff}\left(\frac{x-u}{s}\right)\right\}_{u\in\R,s\in\R_+}\,.
\end{equation}
Invoking Proposition~\ref{PROP: Wff is wavelet} and~\cite[Theorem~6.2]{mallat1999wavelet}, the wavelet transform $\mathcal Wf(u,s)$ of a function $f\in\Ltwo(\R)$ through the space-frequency dictionary $\mathcal D_{ff}$ realizes a second-order {\it multi-scale differential operator}, in the sense that
\begin{equation}\label{EQ: multiscale differentiation}
\lim_{s\to0}\mathcal Wf(u,s)=K s^{5/2} f''(u),
\end{equation}
where $K=\mathcal F(\theta)(0)$. It follows that if $f$ is two-time differentiable in a neighborhood of $u$, then its wavelet transform at $u$ is $\mathcal O(s^{5/2})$, i.e., it is fast decreasing as a function of $s$. Conversely, if $f''(u)$ is unbounded or it is not defined, then the wavelet $\mathcal Wf(u,s)$ decreases at small scales less rapidly than $\mathcal O(s^{5/2})$. In other words, the multi-scale band-pass filtering performed by the wavelet transform associated to the mother wavelet $W_{ff}$
realizes a {\it zoom} into the signal structure by detecting singularities in its second order derivative.
This is a key principle of a {\it wavelet zoom}.

\subsection{The wavelet zoom is not robustly realizable in a feedforward model}

Despite its favorable multi-resolution properties, the feedforward wavelet zoom is not realizable in an analog hardware, being it artificial of biological, without unrealistic topographically precise (and thus not robust) tuning of neural connectivity. To see this, consider a neuron at position $x$ receiving feedforward inputs from two neurons at locations $y_1$ and $y_2$
A change in scale from $1$ to $s$ would imply changing the two associated synaptic strengths by
\[
\Delta W_1=\psi^{ff}_{y_1,s}-\psi^{ff}_{y_1,1}=\frac{1}{\sqrt{s}}W_{ff}\left(\frac{x-y_1}{s}\right)-W_{ff}(x-y_1)\,,
\]
\[
\Delta W_2=\psi^{ff}_{y_2,s}-\psi^{ff}_{y_2,1}=\frac{1}{\sqrt{s}}W_{ff}\left(\frac{x-y_2}{s}\right)-W_{ff}(x-y_2)\,,
\]
respectively. Because the kernel $W_{ff}$ is nonlinear, in general
$\frac{\Delta W_1}{W_{ff}(x-y_1)}\neq\frac{\Delta W_2}{W_{ff}(x-y_2)},$
i.e., the two gains must be scaled by different amounts. In particular at small $s$, this difference can be large even when $y_1$ and $y_2$ are close. This means that changing the scale in a feedforward fashion requires a scale- {\it and} position-dependent finely-tuned scaling of all synaptic gains. There are no biological evidence of such a precise tuning of synaptic gains and the same tuning would be unreliable in sloppy analog hardware.

\subsection{Feedback design of the wavelet zoom}

In this section we show that it is possible to realize the space-frequency dictionary $\mathcal D_{ff}$, defined in~\eqref{EQ: ff dictionary}, through the {\it space-homogeneous} scaling of the damping coefficient $\gamma$, the excitatory connection gain $K_E$, and the inhibitory connection gain $K_I$ in model~\eqref{EQ: V1 neural field generic fourier}, for suitably designed excitatory and inhibitory feedback kernels. To this end, let
\begin{equation}\label{EQ: iso connection ker exc}
W_E(x-y)=\frac{1}{2\alpha}e^{-\frac{|x-y|}{\alpha}}
\end{equation}
and
\begin{equation}\label{EQ: iso connection ker inh}
W_I(x-y)=\frac{1}{2\beta}e^{-\frac{|x-y|}{\beta}}
\end{equation}
with $0<\alpha<\beta$, which is an exponential approximation of the local-excitation/lateral-inhibition feedback motif of spatially extended neural systems~\cite{Wilson1973}. In the spatial frequency domain and for static visual input $h(x,t)\equiv l(x)$, $l\in\Ltwo(\R)$, the resulting neural dynamics~\eqref{EQ: V1 neural field generic fourier} read
\begin{equation}\label{EQ: iso EI neural field fourier}
\hat a_t=\left(-\gamma+\frac{K_E}{1+\alpha^2\lambda^2}-\frac{K_I}{1+\beta^2\lambda^2}\right)\hat a + \hat W_{ff}\hat l\,.
\end{equation}
\begin{lemma}\label{LEM: iso stability}
Dynamics~\eqref{EQ: iso EI neural field fourier} are exponentially stable (in the sense of~\cite[Definition~3]{Bamieh2002}) provided $\gamma>0$ and
$\frac{K_E}{K_I}\frac{\beta^2}{\alpha^2}\leq 1.$
\end{lemma}
\begin{proof}
	\eqref{EQ: iso EI neural field fourier} is exponentially stable provided~\cite[Theorem~1]{Bamieh2002} that $\gamma-\frac{K_E}{1+\alpha^2\lambda^2}+\frac{K_I}{1+\beta^2\lambda^2}\geq C >0$ for all $\lambda\in\R$. This evidently holds with $C=\gamma$ if
$\frac{K_E}{K_I}\frac{1+\beta^2\lambda^2}{1+\alpha^2\lambda^2}\leq 1$ for all $\lambda\in\R$.
Recalling that $0<\alpha<\beta$, we have that $
1\leq\frac{1+\beta^2\lambda^2}{1+\alpha^2\lambda^2}<\frac{\beta^2}{\alpha^2}$  for all $\lambda\in\R$
and the result follows.
\end{proof}
The next theorem provides sufficient conditions for the existence of and explicitly constructs the space-homogeneous, scale-dependent tuning of damping coefficient, and excitatory and inhibitory feedback gains to realize the space-frequency dictionary $\mathcal D_{ff}$ as the spatial transfer function~\eqref{EQ: neural field TF} of model~\eqref{EQ: iso EI neural field fourier}.

\begin{theorem}\label{THM: iso realization}
Let $W_{ff}$, $W_E$, $W_I$ be defined as in~\eqref{EQ: ff kernel shape},~\eqref{EQ: iso connection ker exc}, and~\eqref{EQ: iso connection ker inh}, respectively, with $a=\alpha$ and $b=\beta$. Let $\rho(s)=s^{3/2}$,
\begin{align*}
\kappa_e(s)&=\frac{\alpha^2(s^{-5/2}-s^{-1/2})-\beta^2(s^{-1/2}-s^{3/2})}{\beta^2-\alpha^2}\,,\\
\kappa_i(s)&=\frac{\beta^2(s^{-5/2}-s^{-1/2})-\alpha^2(s^{-1/2}-s^{3/2})}{\beta^2-\alpha^2}\,.
\end{align*}
The spatial transfer function~\eqref{EQ: neural field TF} satisfies
\[
H_{\rho(s),k_e(s),k_i(s)}(\lambda)=\hat\psi^{ff}_s(\lambda)
\]
where $\hat\psi^{ff}_s(\lambda)=\mathcal F(\psi^{ff}_{0,s})(\lambda)$. Furthermore, dynamics~\eqref{EQ: iso EI neural field fourier} with $\gamma=\rho(s)$, $K_E=\kappa_e(s)$, $K_I=\kappa_i(s)$ are exponentially stable for all $0<s\leq 1$.
\end{theorem}
\begin{proof}
Start by observing that, for $a=\alpha$, $b=\beta$,
\begin{align*}
\hat\psi^{ff}_s(\lambda)&=\frac{(N(\alpha,\beta))^{-1}(\beta^2-\alpha^2)\lambda^2}{s^{-5/2}+(\alpha^2+\beta^2)s^{-1/2}\lambda^2+\alpha^2\beta^2s^{3/2}\lambda^4}\,.
\end{align*}
On the other hand,\\
$H_{\gamma,K_E,K_I}(\lambda)=(N(\alpha,\beta))^{-1}(\beta^2-\alpha^2)\lambda^2\Big(\gamma-K_E+K_I+
(\gamma(\alpha^2+\beta^2)
-K_E\beta^2+K_I\alpha^2)\lambda^2+\gamma\alpha^2\beta^2\lambda^4\Big)^{-1}$.\\
Equating the monomial coefficients in the denominators leads to the system of equations
\begin{align*}
\gamma-K_E+K_I&=s^{-5/2}\\
\gamma(\alpha^2+\beta^2)-K_E\beta^2+K_I\alpha^2&=(\alpha^2+\beta^2)s^{-1/2}\\
\gamma\alpha^2\beta^2&=a^2\beta^2s^{3/2}\,
\end{align*}
whose solution is $\gamma=\rho(s)$, $K_E=\kappa_e(s)$, and $K_I=\kappa_i(s)$. To show stability of the resulting spatio-temporal dynamics, observing that $\rho(s)>0$ and invoking
Lemma~\ref{LEM: iso stability}, it remains to show that
\begin{equation}\label{EQ: stability proof}
\kappa(s):=\frac{\kappa_e(s)}{\kappa_i(s)}\leq\frac{\alpha^2}{\beta^2}
\end{equation}
for all $0<s\leq 1$. The rational function $\kappa(s)$ is smooth on $(0,1)$, indeed, its denominator vanishes for $|s|=1$, $|s|=\frac{\beta}{\alpha}>1$, and
\begin{equation}\label{EQ: gain limits}
\lim_{s\to 0}\kappa(s)=\frac{\alpha^2}{\beta^2},\quad \lim_{s\to 1}\kappa(s)=\frac{\beta^2-\alpha^2}{\beta^2}\frac{\alpha^2}{\alpha^2-\beta^2}\,.
\end{equation}
Furthermore, it is easy to verify that $k'(s)<0$ for $s\in(0,1)$ and~\eqref{EQ: stability proof} follows.
\end{proof}

\subsection{The design is robust to approximations and uncertainties}

\begin{figure*}[t!]
	\centering
	\includegraphics[width=0.75\textwidth]{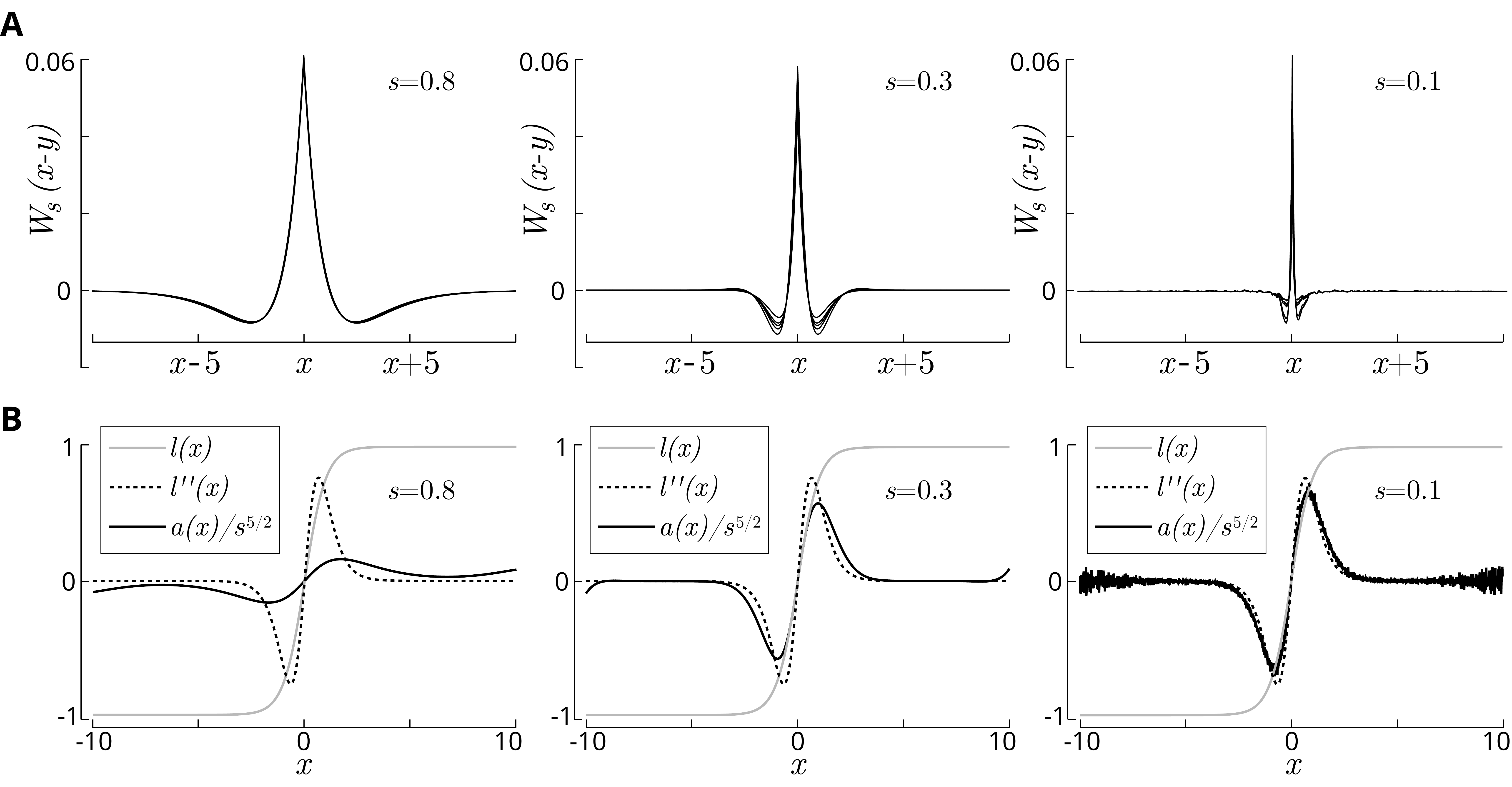}
	\caption{A. Closed-loop kernel $W_s(x-y)=\mathcal F^{-1}(H_{\rho(s),k_e(s),k_i(s)})(x-y)$ with $\alpha=1$ and $\beta=2$ at three different scales and under random perturbations as specified in the text. B. Closed-loop filter spatial response at threes scales and under random perturbations.}\label{FIG: real robu}
\end{figure*}

Our wavelet-zoom feedback design is naturally robust to approximation and uncertainties that might arise in practical applications. We illustrate this fact through a numerical example. For instance, it might be difficult to jointly tune the two feedback gains exactly accordingly the complicated functions $\kappa_e(s)$ and $\kappa_i(s)$ defined in Theorem~\ref{THM: iso realization}. However, as show in~\eqref{EQ: gain limits}, at small scales, the two feedback gains can be scaled proportionally, i.e.
$
\kappa_e(s)\approx\frac{\alpha^2}{\beta^2}\kappa_i(s)\,.
$
In our numerical realization we thus approximate
\[
\kappa_e(s)=\delta\frac{\alpha^2}{\beta^2}\kappa_i(s),\quad 0<\delta\lesssim 1\,.
\]
The factor $\delta$ enforces that the stability condition derived in Lemma~\ref{LEM: iso stability} is robustly satisfied. We also add global small ($10^{-2}$ of the nominal values) random uncertainties to the spatial spread of the various kernels and to the scale parameter $s$ passed to the three functions $\rho(s),\kappa_e(s),\kappa_i(s)$. Finally, we add local perturbations to all local coupling gain, i.e., in our numerical discretization on the one-dimensional mesh $\{x_1,\ldots,x_{Nspace}\}$, we perturb $W_E(x_i-x_j)$ to $W_E(x_i-y_j)(1+\varepsilon p_{ij})$, where $p_{ij}$ is drawn from a normal distribution and $\varepsilon=10^{-4}$.
Each panel of Figure~\ref{FIG: real robu}A shows five realizations of the resulting closed-loop kernel
$
W_s(x-y)=\mathcal F^{-1}(H_{\rho(s),k_e(s),k_i(s)})(x-y)
$,
under random perturbations as specified above and at three different scales $s=0.8,0.3,0.1$. The effect of perturbations becomes noticeable at small scales, but the realization is robust. Future work will aim at quantify precise robustness bounds. Each panel of Figure~\ref{FIG: real robu}B shows the input-output performance of the perturbed closed-loop filter at the same scales. The gray trace shows the input signal ($l(x)=\tanh(x)$ in this example). The black dashed trace shows the input signal second derivative. The black trace is the filter response $a(x)$ scaled by $s^{-5/2}$. Despite perturbations becoming evident at small scales, the scaled filter response robustly converge to the signal's second derivative as predicted by wavelet-zoom theory~\eqref{EQ: multiscale differentiation}.

\section{Discussion and extensions}
\label{SEC: discussion}

\subsection{Realizing spatial wavelet zooms via feedback interconnection of simple spatial transfer functions and gain modulation}

Theorem~\ref{THM: iso realization} shows that it is possible to realize a {\it whole} continuous space-frequency atom dictionary on a fixed kernel interconnection topology, that is, without reshaping the feedforward and feedback kernels, but solely modulating the feedback gains, homogeneously in space, and robustly to uncertainties and spatial heterogeneities. How can this theorem be used in practice? The key idea is that the feedforward and feedback kernels can be hard-wired once and for all (e.g., in an analog DNF architecture~\cite{indiveri2019importance}). Only their gains need to be modulated and this modulation does not need to be fine-tuned in space. A space-homogeneous balanced up-scaling of excitatory and inhibitory interconnection gains suffices to change the closed-loop filter resolution.

\subsection{Realizing richer space-frequency dictionaries}

Although focused on a specific feedforward kernel and associated space-frequency atom dictionary, the spatial frequency-domain closed-loop design methodology spelled out in this paper is general. Namely, any kernel with a rational transfer function is a good candidate to realize its associated space-frequency dictionary by spatial loop-shaping. For instance, the unisotropic case
$
\hat W_{ff}(\lambda)=\frac{i\lambda}{(1+\alpha^2\lambda^2)(1+\beta^2\lambda^2)}
$
can be treated along exactly the same lines of the isotropic case considered here. Richer (e.g., Gabor-like) kernels and associated dictionaries can also be realized taking inspiration from the multi-layer structure of the primary visual cortex~\cite{alonso2002}. The primary visual cortex is indeed composed of multiple feedforward, feedback, and horizontal pathways, of which the ones considered in Section~\ref{SEC: V1 feedback filter} are just the most basic approximation.


\subsection{The role of nonlinearities}

Biological neural dynamics are intrinsically nonlinear, both in time and space. The role of nonlinearities is to transform linear instability into multi-scale excitability and pattern formation~\cite{Franci2016,Drion2018a}. These all-or-none responses robustify information processing by providing neuronal communication with a mixed analog/digital nature~\cite{Drion2018,Sepulchre2018}. It will be crucial to understand how the linear information processing level studied here interacts with the nonlinear nature of neural dynamics, as proposed in~\cite{franci2019sensitivity} for the purely temporal case.

\subsection{Time-varying visual stimuli}

Actual visual stimuli are time-varying. A fundamental extension of the ideas proposed here will be to build time-and-space-invariant systems that realizes space-time-frequency dictionaries capable of performing a wavelet zoom on incoming signal both in time and space. The design procedure could approximate the system response via a separable transfer function, much in the same way as~\cite{Gorinevsky2008}, which would allow to independently design the temporal and spatial responses by joining the theory developed here with classical works on analog temporal wavelet realization such as~\cite{haddad2005log}.

\bibliographystyle{IEEEtran}
\bibliography{CDCwavelet}

\end{document}